\documentclass{amsart}

\usepackage{color}
\usepackage{enumerate}
\usepackage{amsthm,amssymb}
\usepackage{amsmath}
\usepackage{graphicx,epsfig}
\usepackage{hyperref}

\def\R{\mathbb R}

\def\N{\mathbb N}

\newcommand{\norm}[1]{\left\Vert#1\right\Vert}

\newtheorem{theorem}{Theorem}[section]
\newtheorem{lemma}[theorem]{Lemma}
\newtheorem{proposition}[theorem]{Proposition}
\newtheorem{corollary}[theorem]{Corollary}

\theoremstyle{definition}
\newtheorem{definition}[theorem]{Definition}

\theoremstyle{remark}
\newtheorem{remark}[theorem]{Remark}

\numberwithin{equation}{section}

\begin{document}

\title{Kernel estimates for a class of fractional Kolmogorov operators}

\author{Marianna Porfido}
\address{Institute of Applied Analysis, TU Bergakademie Freiberg, Akademiestrasse 6, D-09599 Freiberg, Germany}
\curraddr{}
\email{Marianna.Porfido@math.tu-freiberg.de}
\thanks{}

\author{Abdelaziz Rhandi}
\address{Department of Mathematics, University of Salerno
Via Giovanni Paolo II, 132, I 84084 Fisciano (SA) Italy}
\curraddr{}
\email{arhandi@unisa.it}
\thanks{The authors are members of the ''Gruppo Nazionale per l’Analisi Matematica, la Probabilità e le loro Applicazioni (GNAMPA)'' of the Istituto Nazionale di Alta Matematica (INdAM). This article is based upon work from the project PRIN2022 D53D23005580006 “Elliptic and parabolic problems, heat kernel estimates and spectral theory”.}

\author{Cristian Tacelli}
\address{Department of Mathematics, University of Salerno
Via Giovanni Paolo II, 132, I 84084 Fisciano (SA) Italy}
\curraddr{}
\email{ctacelli@unisa.it}
\thanks{}


\date{}

\dedicatory{}

\begin{abstract}
Assuming a weighted Nash type inequality for the generator $-A$ of a Markov semigroup, we prove a weighted Nash type inequality for its fractional power and deduce non-uniform bounds on the transition kernel corresponding to the Markov semigroup generated by $-A^\alpha$.
\end{abstract}

\maketitle

\section{Introduction}

Let $(X,\mu)$ be a measure space with $\sigma$-finite measure $\mu$ and let $T(\cdot)$ be a symmetric Markov semigroup on $L^2(\mu)$ with generator $-A$, namely a symmetric positivity preserving and $L^\infty$-contractive strongly continuous semigroup in $L^2(\mu)$.

It is well known that for every $t>0$ the operators $T(t)$ are contractions in $L^p(\mu)$ for all $1\leq p\leq \infty$.
One can prove that the Dirichlet form associated with $T(\cdot)$ is non-negative, so
\[
(Af,f)=\int_X f Af\, d\mu\geq 0
\]
for any $f\in D(A)$. This, together with the symmetry of the semigroup $T(\cdot)$, implies that $A$ is a non-negative self-adjoint operator on $L^2(\mu)$, see \cite[Chapter I]{Davies}. We note here that it suffices to consider real-valued functions since positive semigroups are real.\\
Very often the semigroup $T(\cdot)$ is represented by a kernel density with respect to the measure $\mu$, namely for any $t>0$ there exists a non-negative symmetric function $p(t,x,y)$ on $X\times X$ such that
\[
T(t)f(x)=\int_X f(y)p(t,x,y)\, d\mu(y)
\]
for $\mu$ almost every $x\in X$.
For further properties about symmetric Markov semigroups and their generators we address the reader to \cite{Davies} and \cite{Ouhabaz}.

Nash type inequalities are a powerful tool to show the existence of the kernel density $p(t,x,y)$ and to obtain control on it.
The model case is the classical Nash inequality, which in our context may be written as follows
\[
\norm{f}_2^{1+\frac{n}{2}}\leq C\norm{f}_1 (Af,f)^{\frac{n}{4}},
\]
where $n$ is a positive parameter. As shown in \cite[Chapter 2]{Davies}, the previous inequality is equivalent to the classical uniform bound $\norm{T(t)f}_\infty\leq C'^2 t^{-n/2} \norm{f}_1$ for $t>0$, which implies the following uniform bound on the kernel density
\[
p(t,x,y)\leq C'^2 t^{-n/2}
\]
for any $x,y\in X$ and $t>0$.
In the last few years, more general Nash inequalities together with Sobolev inequalities and super-Poincaré inequalities have been studied in order to obtain estimates for the kernel density (see \cite{Bakry}, \cite{Bakry-2}, \cite{BCLS}, \cite{CKKW}, \cite{CKS}, \cite{Co}, \cite{W}).
For instance in \cite{Bakry} the authors introduce weight functions $V$ that are Lyapunov functions for the generator $-A$  (i.e. $0\leq V\in D(A)$ and $-AV\leq cV$ for some non-negative constant $c$ called Lyapunov constant). They deal with Nash type inequalities of the form 
\begin{equation}\label{weighted-Nash Bakry}
\phi\left( \frac{\norm{f}_2^2}{\norm{fV}_1^2}\right) \leq \frac{(A f,f)}{\norm{fV}_1^2},
\end{equation}
where $\phi$ is a positive function defined on $(M,\infty)$ with $x\mapsto \phi(x)/x$ non decreasing and $M$ be a non-negative real number. We point out that if we set $\phi(x)=C x^{1+2/n}$, $M=0$ and $V=1$, then \eqref{weighted-Nash Bakry} is equalivalent to the classical Nash inequality mentioned above.
Under suitable assumptions on the Lyapunov function $V$ and the rate function $\phi$, it is proved in \cite[Theorem 3.5]{Bakry} that \eqref{weighted-Nash Bakry} implies that $\norm{T(t)f}_2\leq K(2t) e^{ct} \norm{fV}_1$ for $t>0$, $f\in L^2(\mu)$ and for a positive function $K$. From that, \cite[Corollary 3.7]{Bakry} gives the following estimate for the kernel density:
\[
p(t,x,y)\leq K(t)^2e^{ct} V(x)V(y)
\]
for all $t>0$ and $\mu \otimes \mu$ almost every $x,y\in X$.
In this paper we will consider a slightly different version of inequality \eqref{weighted-Nash Bakry}, that is 
\begin{equation}\label{intro: fractional Nash ineq for 0<alpha<1}
\norm{f}_2^2 B\left( \frac{\norm{f}_2^2}{\norm{fV}_1^2}\right) \leq (Af,f) + c\norm{f}_2^2,
\end{equation}
where $B:[0,+\infty)\to [0,+\infty)$ is a non-decreasing function and $c$ is the Lyapunov constant associated to the weight function $V$.
We will refer to \eqref{intro: fractional Nash ineq for 0<alpha<1} as \emph{weighted Nash inequality}.

In this work we are interested in establishing estimates for the kernel associated with the fractional operator $-A^\alpha$ for $\alpha>0$. In the last few years this question has received a lot of attention. We mention here \cite{Kinzebulatov-Semenov}, \cite{BDK}, \cite{CKSV}, \cite{JW}, \cite{KSS21}, where specific operators were considered. The idea is to follow the approach described above making use of weighted Nash inequalities. However, in concrete examples it is not easy to directly prove such estimates for the fractional operator $A^\alpha$. Therefore, in Section \ref{Section: weghted-fractional-Nash}, Theorem \ref{Thm: fractional Nash for 0<alpha<1} shows how to obtain a weighted Nash inequality for $A^\alpha$ assuming that \eqref{intro: fractional Nash ineq for 0<alpha<1} holds for the operator $A$. We highlight that the non-weighted case with $c=0$ was treated in \cite{Bendikov-Maheux}.

With this weighted Nash inequality for the fractional operator at hand, in Section \ref{Section: Kernel estimates} we start from the weighted Nash inequality \eqref{intro: fractional Nash ineq for 0<alpha<1} to get estimates for the kernel $p_\alpha(t,x,y)$ associated with the fractional operator $-A^\alpha$. We prove kernel estimates of the following form
\[
p_\alpha(t,x,y)\leq \tilde{K}(t,\phi,c) V(x)V(y)
\]
for all $t>0$ and $\mu \otimes \mu$ almost every $x,y\in X$.
These results are contained in Theorem \ref{Thm: fractional kernel estimate for 0<alpha<1} for the case $0<\alpha<1$ and in Theorem \ref{Thm: fractional kernel estimate for alpha>1} for $\alpha\geq 1$.
Moreover, as soon as $V$ belongs to $L^2(\mu)$, then the kernel density $p_\alpha(t,x,y)$ is in $L^2(\mu\otimes\mu )$. As a consequence, for any $t>0$ the operator $T_\alpha(t)$ is Hilbert-Schmidt on $L^2(\mu)$ and so, it has a discrete spectrum.

Finally, in Section \ref{Section: applications} we illustate our results by applying them to the operator 
\[
A=-\Delta -\nabla \log\rho\cdot\nabla
\]
on the space $L^2(\mu)$ with the measure $d\mu(x)=\rho(x)dx$ and the weight function $V=\rho^{-1/2}$, where $\rho$ is a positive smooth function on $\R^d$. 
Under suitable assumptions, a weighted Nash inequality for $A$ was already proven in \cite[Section 4]{Bakry} in order to obtain estimates for the kernel associated with the operator $-A$ (see \cite[Corollary 4.2]{Bakry}). Applying Theorem \ref{Thm: fractional kernel estimate for 0<alpha<1},
we are now able to find a similar result for the kernel $p_\alpha(t,x,y)$ associated with the fractional operator $-A^\alpha$ for $0<\alpha<1$, i.e.
\[
p_\alpha(t,x,y) \leq C t^{-\frac{d}{2\alpha}} e^{c^\alpha t} \rho^{-1/2}(x) \rho^{-1/2}(y)
\]
for all $t>0$ and $x,y\in\R^d$.
In particular, we consider examples of Cauchy measures $\rho(x)=(1+|x|^2)^{-\beta}$ with $\beta>d$ and exponential measures $\rho(x)=e^{-(1+|x|^2)^{a/2}}$ with $0<a<2,\,\rho(x)=e^{-|x|^a}$ with $a\ge 2$.

\subsection*{Notation}
If $u: \R^d\to \R$, we use the following notation:
\begin{align*}
D_iu=\frac{\partial u}{\partial x_i}, \
D_{ij}u=D_iD_ju,\
\nabla u=&(D_1u, \dots, D_du)
\end{align*}
and
$$
|\nabla u|^2=\sum_{j=1}^d |D_j u|^2 , \qquad
|D^2u|^2=\sum_{i,j=1}^d|D_{ij} u|^2.
$$
Moreover, we denote by $L^2(\mu)$ the classical $L^2$-space on $(X,\mu)$ endowed with the norm $\norm{\cdot}_2$ and the scalar product $(\cdot, \cdot)$.

\section{Weighted Nash type inequalities for fractional powers of non-negative self-adjoint operators} \label{Section: weghted-fractional-Nash}

In this section we prove weighted Nash inequalities for a class of Kolmogorov fractional operators.

Let $(X,\mu)$ be a measure space with $\sigma$-finite measure $\mu$ and let $A$ be a non-negative self-adjoint operator on $L^2(\mu)$. 
We may write $A$ by means of its spectral decomposition as follows
\begin{equation}\label{spectral decomposition A}
A=\int_0^{+\infty} \lambda\, d E_\lambda.
\end{equation}
For $\alpha>0$ we define the fractional power $A^\alpha$ of $A$ by 
\[
(A^\alpha f, f)=\int_0^{+\infty} \lambda^\alpha\, d (E_\lambda f,f)
\]
on the domain $D(A^\alpha)=\{ f\in L^2(\mu)\mid \int_0^{+\infty} \lambda^\alpha d (E_\lambda f,f)<\infty\}$. We have that $A^\alpha$ is a non-negative self-adjoint operator. Moreover, $-A^\alpha$ generates a strongly continuous semigroup of contractions $T_\alpha(\cdot)$ given by
\begin{equation}\label{fractional semigroup}
T_\alpha (t)=\int_0^{+\infty} T(s) d\mu_t^\alpha (s)
\end{equation}
for all $t>0$, where $(\mu_t^\alpha)$ is such that
\[
\int_0^{+\infty} e^{-s\lambda} d\mu_t^\alpha (s)=e^{-t\lambda^\alpha}
\]
for any $\lambda>0$.
In particular, for $\alpha=1/2$, we denote by $P(\cdot)$ the semigroup generated by $-A^{1/2}$.

If $0<\alpha<1$, by the Balakrishnan representation formula (see for example \cite[Chapter 1, Section 9]{Goldstein} and \cite[Proposition 3.1.12]{Haase}), we have 
\begin{equation}\label{Balakrishnan formula}
A^\alpha f= \frac{\sin \alpha \pi}{\pi} \int_0^{+\infty} t^{\alpha-1} (t+A)^{-1} (Af)\, dt
\end{equation}
for all $f\in D(A)$.

We now consider a symmetric Markov semigroup $T(\cdot)$ acting on $L^2(\mu)$ with generator $-A$, namely a family of linear operators mapping satisfying
the following properties:
\begin{enumerate}
\item Preservation of positivity: if $f\geq 0$ then $T(t)f\geq 0$ for all $t\ge 0$.
\item $L^\infty$-contractivity: $\norm{T(t)f}_\infty\leq \norm{f}_\infty$ for all $t\ge 0$ and $f\in L^2(\mu)\cap L^\infty(\mu)$.
\item Semigroup property: $T(0)=Id$ and $T(t+s)=T(t)T(s)$ for all $t,s\ge 0$.
\item Symmetry: $T(t)$ maps $L^2(\mu)$ into itself and $\int_X T(t)fg\, d\mu=\int_X fT(t)g\, d\mu$ for all $f,g\in L^2(\mu)$ and $t\ge 0$.
\item Continuity at $t=0$: $T(t)f\to f$ as $t\to 0$ in $L^2(\mu)$.
\end{enumerate}

We first introduce, as in \cite{Bakry}, Lyapunov functions for the operator $-A$.

\begin{definition} 
A Lyapunov function for the operator $-A$ is a positive function $V\in D(A)$ such that
\[
-AV\leq cV
\]
for a real constant $c$, called the Lyapunov constant.
\end{definition}

\begin{remark}\label{Rem: Lyapunov functions for powers}
Fix $0<\alpha<1$. Assume that $V$ is a Lyapunov function for the operator $-A$ with zero Lyapunov constant. 
Then, $V\in D(A^\alpha)$ (see \cite[Proposition 3.1.1]{Haase}).
Moreover, since $T(\cdot)$ is a positive semigroup, we have that the resolvent operator $(t+A)^{-1}$ is positive for all $t>0$, cf. \cite[Corollary 12.10]{BMR17}. Hence, combining the inequality $-AV\leq 0$ with the Balakrishnan representation formula \eqref{Balakrishnan formula}, it follows that $V$ is a Lyapunov function for the operator $-A^\alpha$ with zero Lyapunov constant.
\end{remark}

The rest of this section is devoted to the proof of the following two results that show how to obtain weighted Nash inequalities for the fractional operator $A^\alpha$ starting from similar inequalities for the operator $A$. For that we need to distinguish between the cases of $0<\alpha<1$ and $\alpha\geq 1$.

\begin{theorem}\label{Thm: fractional Nash for 0<alpha<1}
Let $(X,\mu)$ be a measure space with $\sigma$-finite measure $\mu$ and let $T(\cdot)$ be a symmetric Markov semigroup acting on $L^2(\mu)$ with generator $-A$. Assume that there exist a non-decreasing function $B:[0,+\infty)\to [0,+\infty)$ and a Lyapunov function $V$ for the operator $-A$ in $L^2(\mu)$ with Lyapunov constant $c\geq 0$ such that the following weighted Nash inequality holds
\begin{equation}
\norm{f}_2^2 B\left( \frac{\norm{f}_2^2}{\norm{fV}_1^2}\right) \leq (Af,f) + c\norm{f}_2^2
\end{equation}
for all $f\in D(A)$ with $\norm{fV}_1<\infty$.
Then for every $0<\alpha<1$, there exists $\gamma>0$ 
such that
\begin{equation}\label{fractional Nash ineq for 0<alpha<1}
\gamma \norm{f}_2^2 \left[B\left( \gamma \frac{\norm{f}_2^2}{\norm{fV}_1^2}\right)\right]^\alpha \leq (A^\alpha f,f)+c^\alpha \norm{f}_2^2
\end{equation}
for all $f\in D(A^\alpha)$ with $\norm{fV}_1<\infty$.
\end{theorem}

\begin{theorem}\label{Thm: fractional Nash for alpha>1}
Let $(X,\mu)$ be a measure space with $\sigma$-finite measure $\mu$ and let $A$ be a non-negative self-adjoint operator in $L^2(\mu)$. Assume that there exist a positive function $V$ and a non-decreasing function $B:[0,+\infty)\to [0,+\infty)$ such that the following weighted Nash inequality holds
\begin{equation}\label{eq2: fractional Nash for alpha>1}
\norm{f}_2^2 B\left( \frac{\norm{f}_2^2}{\norm{fV}_1^2}\right) \leq (Af,f)
\end{equation}
for all $f\in D(A)$ with $\norm{fV}_1<\infty$.
Then for every $\alpha\geq 1$ we have
\begin{equation}
\norm{f}_2^2 \left[B\left(\frac{\norm{f}_2^2}{\norm{fV}_1^2}\right)\right]^\alpha \leq (A^\alpha f,f)
\end{equation}
for all $f\in D(A^\alpha)$ with $\norm{fV}_1<\infty$.
\end{theorem}

To prove Theorem \ref{Thm: fractional Nash for 0<alpha<1} we start by considering the case where $\alpha =\frac{1}{2}$.
 
\begin{proposition}\label{Thm: case alpha=1/2}
Let $T(\cdot)$ be a symmetric Markov semigroup acting on $L^2(\mu)$ with generator $-A$. Assume that there exist a non-decreasing function $B:[0,+\infty)\to [0,+\infty)$
and a Lyapunov function $V$ for the operator $-A$ in $L^2(\mu)$ with zero Lyapunov constant such that the following weighted Nash inequality holds
\begin{equation}\label{case 1/2: weighted Nash inequality}
\norm{f}_2^2 B\left( \frac{\norm{f}_2^2}{\norm{fV}_1^2}\right) \leq (Af,f)
\end{equation}
for all $f\in D(A)$ with $\norm{fV}_1<\infty$. Then, for all $\varepsilon\in (0,1)$, we have
\begin{equation}\label{case 1/2}
(1-\varepsilon^2)^{1/2} \norm{f}_2^2 \left[B\left(\varepsilon  \frac{\norm{f}_2^2}{\norm{fV}_1^2}\right)\right]^\frac{1}{2} \leq (A^{1/2}f,f)
\end{equation}
for all $f\in D(A^{1/2})$ with $\norm{fV}_1<\infty$.
\end{proposition}

\begin{proof}
Let $f\in D(A^{1/2})$ with $\norm{fV}_1<\infty$ and let $\varepsilon\in (0,1)$. Without loss of generality we can assume that $f\geq 0$. Indeed otherwise we write the argument for $|f|$ obtaining that
\begin{equation*}
(1-\varepsilon^2)^{1/2} \norm{f}_2^2 \left[B\left(\varepsilon  \frac{\norm{f}_2^2}{\norm{fV}_1^2}\right)\right]^\frac{1}{2} \leq (A^{1/2}|f|,|f|).
\end{equation*}
Then, since the semigroup $P(\cdot)$ generated by $-A^{1/2}$ is positive, we have, for every $f\in D(A^{1/2}),\,|f|\in D(A^{1/2})$, and $(A^{1/2}|f|,|f|)\leq (A^{1/2}f,f)$ by \cite[Theorem 1.3.2]{Davies} or \cite[Theorem 2.7]{Ouhabaz}. Thus, \eqref{case 1/2} holds true.

We now prove the theorem assuming that $f\geq 0$. We observe that, since $P(\cdot)$ is a symmetric Markov semigroup, then $P(t)f\geq 0$ for any $t\ge 0$.
First we show that
\begin{equation}\label{eq1: case alpha=1/2}
\norm{VP(t)f}_1\leq \norm{fV}_1
\end{equation}
for all $t>0$. Since, by Remark \ref{Rem: Lyapunov functions for powers}, $V$ is a Lyapunov function for the operator $-A^{1/2}$ with zero Lyapunov constant, i.e. $-A^{1/2}V\leq 0$. Hence, we get
\begin{equation}\label{eq7: case alpha=1/2}
\frac{d}{dt}\norm{VP(t)f}_1=\int(-A^{1/2}P(t)f)V\, d\mu = \int (P(t)f) (-A^{1/2}V)\, d\mu \leq 0.
\end{equation}
This proves the inequality \eqref{eq1: case alpha=1/2}.
Moreover, since $P(t)f\in D(A)$ for any $t>0$ and $\norm{VP(t)f}_1<\infty$, applying \eqref{case 1/2: weighted Nash inequality} to the function $P(t)f$ we find that
\begin{equation}\label{eq2: case alpha=1/2}
\norm{P(t)f}_2^2 B\left( \frac{\norm{P(t)f}_2^2}{\norm{VP(t)f}_1^2}\right) \leq (AP(t)f,P(t)f).
\end{equation}
We set 
\[
\phi(t)=\frac{\norm{P(t)f}_2^2}{\norm{fV}_1^2},\quad t\ge 0.
\]
Given that $AP(t)f=\frac{d^2}{dt^2 }P(t)f,\,t>0$, we have
\[
\dot{\phi}(t)=\frac{d}{dt} \phi(t)=-2\frac{(A^{1/2}P(t)f,P(t)f)}{\norm{fV}_1^2},\quad t\ge 0,
\]
and 
\[
\ddot{\phi}(t)=\frac{d^2}{dt^2} \phi(t)= 4\frac{(AP(t)f,P(t)f)}{\norm{fV}_1^2},\quad t>0.
\]
Using first the inequality \eqref{eq2: case alpha=1/2} and then combining \eqref{eq1: case alpha=1/2} with the fact that $B$ is a non-decreasing function, we deduce that
\begin{equation}\label{eq3: case alpha=1/2}
\ddot{\phi}(t)\geq 4 \phi(t)B\left( \frac{\norm{P(t)f}_2^2}{\norm{VP(t)f}_1^2}\right)\geq 4 \phi(t) B(\phi(t))
\end{equation}
for all $t>0$. We observe that $-\dot{\phi}\geq 0$ because  $A^{1/2}$ is a non-negative operator. Multiplying both sides  in \eqref{eq3: case alpha=1/2} by $-\dot{\phi}(t)$ yields
\begin{equation}\label{eq4: case alpha=1/2}
-[\dot{\phi}^2]'(t)\geq -4[\phi^2(t)]' B(\phi(t))
\end{equation}
for any $t>0$.
We now take care of both sides individually. Fix $T>0$. Integrating the left hand side over $[0,T]$ we have
\[
-\int_0^T [\dot{\phi}^2]'(s)\, ds= \dot{\phi}^2(0)-\dot{\phi}^2(T)\leq \dot{\phi}^2(0) = 4\frac{(A^{1/2}f,f)^2}{\norm{fV}_1^4}.
\]
Integrating the right hand side over $[0,T]$ leads to
\[
-4\int_0^T [\phi^2(s)]' B(\phi(s))\, ds= 4 \int_{\phi^2(T)}^{\phi^2(0)}B(\sqrt{x})\, dx.
\]
Then, by \eqref{eq4: case alpha=1/2} we obtain
\begin{equation}\label{eq5: case alpha=1/2}
\frac{(A^{1/2}f,f)^2}{\norm{fV}_1^4}\geq \int_{\phi^2(T)}^{\phi^2(0)}B(\sqrt{x})\, dx.
\end{equation}
Let us assume that 
\begin{equation}\label{eq6: case alpha=1/2}
\lim_{T\to\infty}\norm{P(T)f}_2=0.
\end{equation}
It implies that $\phi^2(T)\to 0$ as $T\to\infty$. Then, taking the limit as $T\to\infty$ in \eqref{eq5: case alpha=1/2} we get
\[
\frac{(A^{1/2}f,f)^2}{\norm{fV}_1^4}\geq \int_{0}^{\phi^2(0)}B(\sqrt{x})\, dx.
\]
From that, since $B$ is non-decreasing, it follows that
\begin{align*}
(1-\varepsilon^2)\phi^2(0)B(\varepsilon \phi(0))&=\int_{\varepsilon^2 \phi^2(0)}^{\phi^2(0)}B(\varepsilon \phi^2(0))\, dx \leq \int_{\varepsilon^2 \phi^2(0)}^{\phi^2(0)} B(\sqrt{x})\, dx\\
&\leq \int_{0}^{\phi^2(0)}B(\sqrt{x})\, dx \leq \frac{(A^{1/2}f,f)^2}{\norm{fV}_1^4}.
\end{align*}
Thus, we derive the inequality \eqref{case 1/2} under the assumption \eqref{eq6: case alpha=1/2}. In order to consider the general case, we apply the argument to the operator $A_\rho=A+\rho I$ with $\rho>0$ and we let $\rho\to 0$.
\end{proof}

Let us now iterate the result obtained in Proposition \ref{Thm: case alpha=1/2}.

\begin{proposition}\label{Prop: iteration}
Under the same assumptions as in Proposition \ref{Thm: case alpha=1/2}, for all $n\in \N\setminus \{ 0\}$ there exist $a_n, b_n>0$ such that
\begin{equation}\label{eq: induction}
a_n \norm{f}_2^2 \left[B\left(b_n \frac{\norm{f}_2^2}{\norm{fV}_1^2}\right)\right]^{1/2^n} \leq (A^{1/2^n}f,f)
\end{equation}
for all $f\in D(A^{1/2^n})$ with $\norm{fV}_1<\infty$.
\end{proposition}

\begin{proof}
We prove the proposition by induction on $n$. For $n=1$ the statement follows from Proposition \ref{Thm: case alpha=1/2}. Assume now that the inequality \eqref{eq: induction} holds true. We observe that $V$ is a Lyapunov function for the operator $-A^{1/2^n}$ with zero Lyapunov constant by Remark \ref{Rem: Lyapunov functions for powers}. Then we can apply Proposition \ref{Thm: case alpha=1/2} with $\tilde{B}=a_n [B(b_n \cdot)]^{1/2^n}$ and $\tilde{A}=A^{1/2^n}$. As a consequence, for all $\varepsilon\in (0,1)$ we have
\[
(1-\varepsilon^2)^{1/2} a_n^{1/2} \norm{f}_2^2 \left[B\left(\varepsilon  b_n\frac{\norm{f}_2^2}{\norm{fV}_1^2}\right)\right]^\frac{1}{2^{n+1}} \leq (A^{1/2^{n+1}}f,f)
\]
for all $f\in D(A^{1/2^{n+1}})$ with $\norm{fV}_1<\infty$.
\end{proof}

The following convexity argument is also needed.

\begin{proposition}\label{Prop: convexity argument}
Let $\Lambda\colon [0,+\infty)\to [0,+\infty)$ be a non-decreasing function and $A$ be a non-negative self-adjoint operator such that
\begin{equation}\label{eq: hp: convexity argument}
\norm{f}_2^2\Lambda\left(\frac{\norm{f}_2^2}{\norm{fV}_1^2}\right) \leq (Af,f)
\end{equation}
for all $f\in D(A)$ with $\norm{fV}_1<\infty$.
Then for any convex non-decreasing and non-negative function $\Phi$ we have
\[
\norm{f}_2^2\,\Phi\circ \Lambda\left(\frac{\norm{f}_2^2}{\norm{fV}_1^2}\right) \leq (\Phi(A)f,f)
\]
for all $f\in D(\Phi(A))$ with $\norm{fV}_1<\infty$.
\end{proposition}

\begin{proof}
Without loss of generality, let us assume $\norm{f}_2=1$. The general case follows by writing the argument below for the function $f/\norm{f}_2$. 
If $\norm{f}_2=1$, then \eqref{eq: hp: convexity argument} becomes
\begin{equation}\label{eq: convexity argument}
\Lambda\left(1/\norm{fV}_1^2\right) \leq (Af,f).
\end{equation}
Since the function $\Phi$ is non-decreasing, applying it to both sides of \eqref{eq: convexity argument} and taking into account the spectral decomposition of the operator $A$ given by \eqref{spectral decomposition A}, we obtain that
\[
\Phi\circ \Lambda\left(1/\norm{fV}_1^2\right) \leq \Phi\left(\int_0^\infty \lambda\, d(E_\lambda f,f)\right).
\]
We now observe that $(E_\lambda f,f)$ is a probability measure because $\norm{f}_2=1$. By the convexity of the function $\Phi$, we apply Jensen's inequality to get the following inequality
\[
\Phi\circ \Lambda\left(1/\norm{fV}_1^2\right) \leq \int_0^\infty \Phi(\lambda)\, d(E_\lambda f,f)= (\Phi(A)f,f)
\]
for all $f\in D(\Phi(A))$ with $\norm{fV}_1<\infty$.
This proves the statement.
\end{proof}

The following comparison between $(A+c)^\alpha$ and $A^\alpha +c^\alpha$ is useful to prove Theorem \ref{Thm: fractional Nash for 0<alpha<1}.

\begin{lemma}\label{Lem: useful inequality}
Let $A$ be a non-negative self-adjoint operator in $L^2(\mu)$ and let $c\geq 0$. Then, for all $f\in D(A^\alpha)$, we have
\[
((A+c)^\alpha f,f)\leq  2^{\alpha-1}[(A^\alpha f,f)+c^\alpha \norm{f}_2^2]
\]
if $\alpha\geq 1$, whereas we have
\[
((A+c)^\alpha f,f)\leq (A^\alpha f,f)+c^\alpha \norm{f}_2^2
\]
if $0<\alpha<1$.
\end{lemma}

\begin{proof}
By the spectral theorem the operator $A$ is unitarily equivalent to a multiplication operator, namely there exists a measure space $(Y,\nu)$ which is $\sigma$-finite, a measurable function $q\colon Y\to\R$ and a unitary map $U\colon L^2(X,\mu)\to L^2(Y,\nu)$ with
\[
A=U^{-1} M_q U.
\]
The multiplication operator $M_q$ on $L^2(Y,\nu)$ is defined by $M_qf(x)=q(x)f(x)$ for $x\in Y$ with domain $D(M_q)=\{ f\in L^2(Y,\nu) \mid qf\in L^2(Y,\nu)\}$.
Hence, by the functional calculus (see \cite[Chapter 2, Section 6]{Goldstein}) the operator $(A+c)^\alpha$ is given by
\[
(A+c)^\alpha=U^{-1} M_{(q+c)^\alpha} U.
\]
Let $f\in D(A^\alpha)$ and assume $\alpha\geq 1$. Since $U$ is a unitary operator, we have
\begin{align*}
((A+c)^\alpha f,f)&=(U^{-1} M_{(q+c)^\alpha} U f,f)=((q+c)^\alpha U f,Uf)\leq 2^{\alpha-1} [ (q^\alpha Uf, Uf)+ c^\alpha (Uf, Uf)]\\
&= 2^{\alpha-1} [ (U^{-1}q^\alpha Uf,f)+ c^\alpha (U^{-1}Uf,f)]=2^{\alpha-1}[(A^\alpha f,f)+c^\alpha \norm{f}_2^2].
\end{align*}
Similarly, for $0<\alpha<1$ we obtain
\[
((A+c)^\alpha f,f)=((q+c)^\alpha U f,Uf)\leq (q^\alpha Uf, Uf)+ c^\alpha (Uf, Uf)
=(A^\alpha f,f)+c^\alpha \norm{f}_2^2.
\]
\end{proof}

\subsection{Proof of Theorem \ref{Thm: fractional Nash for 0<alpha<1}}

Fix $0<\alpha<1$. Let us assume that $c=0$. 
We choose $n\in \N\setminus \{ 0\}$ such that $\alpha_n=1/2^n\leq \alpha$.
Then, by Proposition \ref{Prop: iteration}, there exist $a_n, b_n>0$ such that
\[
a_n \norm{f}_2^2 \left[B\left(b_n \frac{\norm{f}_2^2}{\norm{fV}_1^2}\right)\right]^{\alpha_n} \leq (A^{\alpha_n}f,f)
\]
for all $f\in D(A^{\alpha_n})$ with $\norm{fV}_1<\infty$.
We apply Proposition \ref{Prop: convexity argument} to the operator $A^{\alpha_n}$ with $\Phi(t)=t^{\alpha/\alpha_n}$ and $\Lambda(t)=a_n [B(b_n t)]^{\alpha_n}$. Hence, we obtain that
\[
a_n^{\alpha/\alpha_n} \norm{f}_2^2 \left[B\left(b_n \frac{\norm{f}_2^2}{\norm{fV}_1^2}\right)\right]^{\alpha} \leq (A^{\alpha}f,f)
\]
for all $f\in D(A^{\alpha})$ with $\norm{fV}_1<\infty$.
Since the function $B$ is non-decreasing, it suffices to choose $\gamma$ to be the minimum between $a_n^{\alpha/\alpha_n} $ and $b_n$ to get that
\begin{equation} \label{eq1: fractional Nash for 0<alpha<1}
\gamma \norm{f}_2^2 \left[B\left(\gamma\frac{\norm{f}_2^2}{\norm{fV}_1^2}\right)\right]^\alpha \leq(A^\alpha f,f)
\end{equation}
for all $f\in D(A^{\alpha})$ with $\norm{fV}_1<\infty$. We now consider the general case, so we assume that $V$ is a Lyapunov function for the operator $-A$ with Lyapunov constant $c\geq 0$. Since this implies that $V$ is a Lyapunov function for the operator $-(A+c)$ with zero Lyapunov constant, it is possible to apply \eqref{eq1: fractional Nash for 0<alpha<1} to the operator $A+c$. Therefore, the following inequality holds
\[
\gamma \norm{f}_2^2 \left[B\left( \gamma \frac{\norm{f}_2^2}{\norm{fV}_1^2}\right)\right]^\alpha \leq ((A+c)^\alpha f,f)
\]
for all $f\in D(A^{\alpha})$ with $\norm{fV}_1<\infty$.
Finally, taking into account Lemma \ref{Lem: useful inequality}, inequality \eqref{fractional Nash ineq for 0<alpha<1} follows.\qed

\begin{remark}\label{rem: assumption V in the domain of A}
The proof of inquality \eqref{eq1: case alpha=1/2} and, in particular, the integration by parts in \eqref{eq7: case alpha=1/2} are the only points in which the assumption $V\in D(A)$ plays a role in showing that Theorem \ref{Thm: fractional Nash for 0<alpha<1} is valid.
\end{remark}

\subsection{Proof of Theorem \ref{Thm: fractional Nash for alpha>1}}

Fix $\alpha\geq 1$ and let $f\in D(A^\alpha)$ with $\norm{fV}_1<\infty$. If we apply Proposition \ref{Prop: convexity argument} with $\Phi(t)=t^\alpha$ and $\Lambda=B$, then we obtain that
\[
\norm{f}_2^2 \left[B\left(\frac{\norm{f}_2^2}{\norm{fV}_1^2}\right)\right]^\alpha \leq (A^\alpha f,f).
\]
This proves Theorem \ref{Thm: fractional Nash for alpha>1}.\qed

\begin{remark}\label{rem: fractional Nash inequality for alpha>1}
Fix $\alpha\geq 1$. Let $c$ be a non negative constant and assume the hypotheses of Theorem \ref{Thm: fractional Nash for alpha>1} hold with the following weighted Nash inequality instead of \eqref{eq2: fractional Nash for alpha>1}:
\begin{equation*}
\norm{f}_2^2 B\left( \frac{\norm{f}_2^2}{\norm{fV}_1^2}\right) \leq (Af,f) + c\norm{f}_2^2
\end{equation*}
for all $f\in D(A)$ with $\norm{fV}_1<\infty$.
Then the operator $A+c$ satisfies inequality \eqref{eq2: fractional Nash for alpha>1}. Therefore, by Theorem \ref{Thm: fractional Nash for alpha>1}, we have
\[
\norm{f}_2^2 \left[B\left(\frac{\norm{f}_2^2}{\norm{fV}_1^2}\right)\right]^\alpha \leq ((A+c)^\alpha f,f)
\]
for all $f\in D(A)$ with $\norm{fV}_1<\infty$.
Applying now Lemma \ref{Lem: useful inequality} we conclude that  the following inequality holds
\begin{equation*}
\norm{f}_2^2 \left[B\left(\frac{\norm{f}_2^2}{\norm{fV}_1^2}\right)\right]^\alpha \leq 2^{\alpha-1}\left[(A^\alpha f,f)+c^\alpha \norm{f}_2^2\right]
\end{equation*}
for all $f\in D(A^\alpha)$ with $\norm{fV}_1<\infty$.
\end{remark}

\section{Kernel estimates}\label{Section: Kernel estimates}

In this section we use Theorem \ref{Thm: fractional Nash for 0<alpha<1} to obtain estimates of the kernel of the semigroup generated by $-A^\alpha$. We start with the case $\alpha\in (0,1)$.

\begin{theorem}\label{Thm: norm semigroup for 0<alpha<1}
Let $(X,\mu)$ be a measure space with $\sigma$-finite measure $\mu$ and let $T(\cdot)$ be a symmetric Markov semigroup acting on $L^2(\mu)$ with generator $-A$. Fix $0<\alpha<1$. Assume that there exist a Lyapunov function $V$ for the operator $-A$ in $L^2(\mu)$ with Lyapunov constant $c\geq 0$ and a non-decreasing function $B:[0,+\infty)\to [0,+\infty)$ such that $A$ satisfies the following weighted Nash inequality
\begin{equation}
\norm{f}_2^2 B\left( \frac{\norm{f}_2^2}{\norm{fV}_1^2}\right) \leq (Af,f) + c\norm{f}_2^2
\end{equation}
for all $f\in D(A)$ with $\norm{fV}_1<\infty$.
Moreover, given $\gamma$ as in \eqref{fractional Nash ineq for 0<alpha<1}, 
assume that 
\begin{equation}\label{eq1: norm semigroup for 0<alpha<1}
U(x):=\int_x^\infty \frac{1}{\gamma u [B(\gamma u)]^\alpha}\, du<\infty
\end{equation}
for all $x>0$.
Then the semigroup $T_\alpha(\cdot)$ generated by $-A^\alpha$ satisfies the following inequality
\begin{equation}\label{eq2: norm semigroup for 0<alpha<1}
\norm{T_\alpha(t)f}_2\leq K(2t) e^{c^\alpha t} \norm{fV}_1
\end{equation}
for all $t>0$ and all functions $f\in L^2(\mu)$.
Here, the function $K$ is defined by
\begin{equation}\label{def: function K}
K(x):=\begin{cases}
\sqrt{U^{-1}(x)} & \text{if } 0<x<U(0),\\
0 & \text{if } x\geq U(0).
\end{cases}
\end{equation}
\end{theorem}

\begin{proof}
We first prove the theorem assuming that $c=0$. Then $V$ is a Lyapunov function for $-A^\alpha$ by Remark \ref{Rem: Lyapunov functions for powers} and $\phi(x)=\gamma x [B(\gamma x)]^\alpha$ is a positive function  defined on $(0,+\infty)$ with $\phi(x)/x$ non decreasing and $\int^\infty 1/\phi(x)\, dx<\infty$, by \eqref{eq1: norm semigroup for 0<alpha<1}.
Moreover, by Theorem \ref{Thm: fractional Nash for 0<alpha<1}, the operator $A^\alpha$ satisfies a weighted Nash inequality of this form
\[
\phi\left( \frac{\norm{f}_2^2}{\norm{fV}_1^2}\right) \leq \frac{(A^\alpha f,f)}{\norm{fV}_1^2},
\]
for all $f\in D(A^\alpha)$ such that $\norm{fV}_1<\infty$.
Hence, by \cite[Theorem 3.5]{Bakry} we infer that
\begin{equation}\label{eq5: norm semigroup for 0<alpha<1}
\norm{T_\alpha(t)f}_2\leq K(2t) \norm{fV}_1,
\end{equation}
for all $t>0$ and all functions $f\in L^2(\mu)$.

We now turn to the general case assuming that $c\geq 0$. Since $V$ is a Lyapunov function for the operator $-(A+c)$ with zero Lyapunov constant, from what we proved above we have
\begin{equation}\label{eq3: norm semigroup for 0<alpha<1}
\norm{e^{-t(A+c)^\alpha}f}_2\leq K(2t) \norm{fV}_1,
\end{equation}
for all $t>0$ and all functions $f\in L^2(\mu)$, where $e^{-t(A+c)^\alpha}$ is the semigroup generated by $-(A+c)^\alpha$.
We now prove that 
\begin{equation}\label{eq4: norm semigroup for 0<alpha<1}
\norm{T_\alpha(t)f}_2\leq e^{c^\alpha t}\norm{e^{-t(A+c)^\alpha}f}_2,
\end{equation}
for all $t>0$ and for all $f\in L^2(\mu)$.  
Since $0<\alpha<1$, we apply \cite[Proposition 5.1.14]{Martinez} to infer that there is a bounded self-adjoint operator $T_c$ such that
\[
(A+c)^\alpha=A^\alpha+T_c.
\]
Moreover, Lemma \ref{Lem: useful inequality} implies that
\[
(T_c f,f)=((A+c)^\alpha f,f)-(A^\alpha f,f)\leq c^\alpha \norm{f}_2^2,
\]
for all $f\in D(A^\alpha)$. Hence, we have that $\norm{T_c}\leq c^\alpha$, from which we get that
\[
\norm{e^{tT_c}f}_2\leq e^{c^\alpha t} \norm{f}_2,
\]
for all $t>0$ and for all $f\in L^2(\mu)$, where $e^{tT_c}$ is the semigroup generated by $T_c$.
We finally get \eqref{eq4: norm semigroup for 0<alpha<1} as follows
\[
\norm{T_\alpha(t)f}_2=\norm{e^{t(T_c-(A+c)^\alpha)}f}_2\leq e^{c^\alpha t} \norm{e^{-t(A+c)^\alpha}f}_2,
\]
for all $t>0$ and for all $f\in L^2(\mu)$.  
Combining inequalities \eqref{eq3: norm semigroup for 0<alpha<1} and \eqref{eq4: norm semigroup for 0<alpha<1} we obtain the statement.
\end{proof}

\begin{remark} \label{rem: different U}
Suppose that condition \eqref{eq1: norm semigroup for 0<alpha<1} holds for all $x>M$, for some $M\geq 0$. Then in the proof of Theorem \ref{Thm: norm semigroup for 0<alpha<1} we can still apply \cite[Theorem 3.5]{Bakry} with the function $\phi$ defined on $(M,+\infty)$. As a consequence we obtain that inequality \eqref{eq5: norm semigroup for 0<alpha<1} is satisfied with the function $K$ defined by 
\begin{equation*}
\begin{cases}
\sqrt{U^{-1}(x)} & \text{if } 0<x<U(M),\\
\sqrt{M} & \text{if } x\geq U(M),
\end{cases}
\end{equation*}
and $U$ the function on $(M,+\infty)$ given by
\[
U(x)=\int_x^\infty \frac{1}{\gamma u [B(\gamma u)]^\alpha}\, du.
\]
The rest of the proof carries over verbatim.
\end{remark}

\begin{theorem}\label{Thm: fractional kernel estimate for 0<alpha<1}
Under the same assumptions of Theorem \ref{Thm: norm semigroup for 0<alpha<1}, we have that the semigroup $T_\alpha(\cdot)$ generated by $-A^\alpha$ has a kernel $p_\alpha$ with respect to $\mu$ which satisfies
\begin{equation}\label{eq: Thm: fractional kernel estimate for 0<alpha<1}
p_\alpha (t,x,y)\leq K(t)^2 e^{c^\alpha t} V(x)V(y),
\end{equation}
for all $t>0$ and $\mu \otimes \mu$ almost every $x,y\in X$.
\end{theorem}

\begin{proof}
By Theorem \ref{Thm: norm semigroup for 0<alpha<1} we infer that 
\[
\norm{T_\alpha(t)f}_2\leq K(2t) e^{c^\alpha t} \norm{fV}_1,
\]
for all $t>0$ and for all $f\in L^2(\mu)$. Hence, applying \cite[Proposition 3.1]{Bakry}, we have that the operator $T_\alpha(t)\circ T_\alpha(t)$, that is $T_\alpha(2t)$, is represented by a kernel with respect to $\mu$ satisfying
\[
p_\alpha (2t,x,y)\leq K(2t)^2 e^{2c^\alpha t} V(x)V(y),
\]
for all $t>0$ and $\mu \otimes \mu$ almost every $x,y\in X$.
\end{proof}

We now turn to the case $\alpha\geq 1$.

\begin{theorem}\label{Thm: norm semigroup for alpha>1}
Let $(X,\mu)$ be a measure space with $\sigma$-finite measure $\mu$ and let $T(\cdot)$ be a symmetric Markov semigroup acting on $L^2(\mu)$ with generator $-A$. Fix $\alpha \geq 1$. Assume that there exist a Lyapunov function $V$ for the operator $-A^\alpha $ with Lyapunov constant $c_\alpha\geq 0$ and a non-decreasing function $B:[0,+\infty)\to [0,+\infty)$ such that $A$ satisfies the following weighted Nash inequality
\begin{equation}
\norm{f}_2^2 B\left( \frac{\norm{f}_2^2}{\norm{fV}_1^2}\right) \leq (Af,f) + c\norm{f}_2^2
\end{equation}
for some $c\geq 0$ and for all $f\in D(A)$ with $\norm{fV}_1<\infty$.
Moreover, assume that 
\begin{equation}\label{eq1: norm semigroup for alpha>1}
U(x):=\int_x^\infty \frac{2^{\alpha -1}}{ u [B( u)]^\alpha}\, du<\infty
\end{equation}
for all $x>0$.
Then the semigroup $T_\alpha(\cdot)$ generated by $-A^\alpha$ satisfies the following inequality
\begin{equation}\label{eq2: norm semigroup for alpha>1}
\norm{T_\alpha(t)f}_2\leq \left\{\begin{array}{ll}
K(2t) e^{c_\alpha t} \norm{fV}_1,\,\hbox{\ if }c_\alpha \ge c^\alpha ,\\
K(2t)e^{c^\alpha t} \norm{fV}_1,\,\hbox{\ if }c_\alpha < c^\alpha 
\end{array}
\right.
\end{equation}
for all $t>0$ and all functions $f\in L^2(\mu)$.
Here, the function $K$ is defined by
\begin{equation}
K(x):=\begin{cases}
\sqrt{U^{-1}(x)} & \text{if } 0<x<U(0),\\
0 & \text{if } x\geq U(0).
\end{cases}
\end{equation}
\end{theorem}
\begin{proof}
We apply Theorem \ref{Thm: fractional Nash for alpha>1} and Remark \ref{rem: fractional Nash inequality for alpha>1} to infer that the operator $A^\alpha+ c^\alpha$ satisfies a weighted Nash inequality of this form
\[
\phi\left( \frac{\norm{f}_2^2}{\norm{fV}_1^2}\right) \leq \frac{((A^\alpha + c^\alpha) f,f)}{\norm{fV}_1^2},
\]
for all $f\in D(A^\alpha)$ such that $\norm{fV}_1<\infty$, where the function $\phi$ is defined by $\phi(x)=2^{1-\alpha}x [B(x)]^\alpha$ for $x>0$. 
We observe that $\phi(x)/x$ is non decreasing and that \eqref{eq1: norm semigroup for alpha>1} implies that $\int^\infty 1/\phi(x)\, dx<\infty$.
Furthermore, since $V$ is a Lyapunov function for the operator $-A^\alpha$ with Lyapunov constant $c_\alpha\geq 0$, then $V$ is a Lyapunov function for $-(A^\alpha+ c^\alpha)$ with Lyapunov constant $c_\alpha-c^\alpha$ if $c_\alpha \ge c^\alpha$ and one takes $0$ as a Lyapunov constant when $c_\alpha <c^\alpha$.
Hence, we invoke \cite[Theorem 3.5]{Bakry} to deduce, if we denote by $e^{-t(A^\alpha+c^\alpha)}$ the semigroup generated by $-(A^\alpha+ c^\alpha)$, then
\begin{equation*}
\norm{e^{-t(A^\alpha+c^\alpha)}f}_2\leq \left\{\begin{array}{ll}
K(2t) e^{(c_\alpha-c^\alpha)t}\norm{fV}_1,\,\,\hbox{\ if }c_\alpha \ge c^\alpha ,\\
K(2t)\norm{fV}_1,\,\quad \quad \quad \quad\hbox{\ if }c_\alpha < c^\alpha
\end{array}
\right.
\end{equation*}
for all $t>0$ and all functions $f\in L^2(\mu)$.
Finally, since $e^{-t(A^\alpha+c^\alpha)}=e^{-c^\alpha t}\, T_\alpha(t)$ for all $t>0$, we obtain inequality \eqref{eq2: norm semigroup for alpha>1}.
\end{proof}

\begin{theorem}\label{Thm: fractional kernel estimate for alpha>1}
Under the same assumptions of Theorem \ref{Thm: norm semigroup for alpha>1}, we have that the semigroup $T_\alpha(\cdot)$ generated by $-A^\alpha$ has a kernel $p_\alpha$ with respect to $\mu$ which satisfies
\begin{equation}
p_\alpha (t,x,y)\leq \left\{\begin{array}{ll}
K(t)^2 e^{c_\alpha t} V(x)V(y),\,\,\hbox{\ if }c_\alpha \ge c^\alpha ,\\
K(t)^2 e^{c^\alpha t} V(x)V(y),\,\,\hbox{\ if }c_\alpha < c^\alpha 
\end{array}
\right.
\end{equation}
for all $t>0$ and $\mu \otimes \mu$ almost every $x,y\in X$.
\end{theorem}
\begin{proof}
The same proof of Theorem \ref{Thm: fractional kernel estimate for 0<alpha<1} works if we apply Theorem \ref{Thm: norm semigroup for alpha>1} instead of Theorem \ref{Thm: norm semigroup for 0<alpha<1} and if we replace $c^\alpha$ with $c_\alpha$.
\end{proof}

\section{Some applications}\label{Section: applications}

In this section we aim to apply Theorem \ref{Thm: fractional kernel estimate for 0<alpha<1} to the operator
\begin{equation}\label{example: Kolmogorov operator}
A=-\Delta -\nabla \log\rho\cdot\nabla
\end{equation}
on the space $L^2(\mu)$ with the measure $d\mu(x)=\rho(x)dx$, where $\rho$ is a positive $C^2$-function on $\R^d$. 
By \cite[Corollary 3.7]{Albanese-Lorenzi-Mangino} we have that $
-\overline{A}|_{C_c^\infty(\R^d)}$ on $L^2(\mu)$ generates a symmetric Markov semigroup $T(\cdot)$.
We still denote by $-A$ its generator. 

We first consider the weight function 
\[
V=\rho^{-1/2}.
\]
Even if $V$ is never in $L^2(\mu)$, it will play the role of Lyapunov function.
As proved in \cite[Theorem 4.1]{Bakry}, the classical Nash inequality in $\R^d$
\[
\left(\int_{\R^d} |f|^2 dx\right)^\frac{2+d}{4}\leq C_d \left(\int_{\R^d} |f| dx\right) \left(\int_{\R^d} |\nabla f|^2 dx\right)^\frac{d}{4}
\]
is equivalent to
\[
\norm{f}_2^{2+\frac{4}{d}}\leq C_d^\frac{4}{d} \norm{fV}_1^\frac{4}{d} \left((Af,f)+\int_{\R^d}\frac{-AV}{V}f^2\, d\mu\right)
\]
for all smooth functions $f$ on $\R^d$ with compact support.
Moreover, if $-AV\leq cV$ for some $c\in\R$, then
\[
\norm{f}_2^{2+\frac{4}{d}}\leq C_d^\frac{4}{d} \norm{fV}_1^\frac{4}{d} \left((Af,f)+c\norm{f}_2^2\right).
\]
In other words, $A$ satisfies a weighted Nash inequality such as \eqref{intro: fractional Nash ineq for 0<alpha<1} with $B(x)=x^\frac{2}{d}$.
Consequently, under suitable assumptions on the function $\rho$, in the next result we apply Theorem \ref{Thm: fractional kernel estimate for 0<alpha<1} to estimate the kernel associated with the fractional operator $A^\alpha$ for $0<\alpha<1$.

\begin{corollary}\label{Cor: fractional kernel estimate example 1}
With the above notation, let $0<\alpha<1$ and assume that $\mu$ is a probability measure. 
Moreover, assume that
\begin{enumerate}
\item $V\in L^1(\mu)$, $AV\in L^1(\mu)$ and $-AV\leq cV$ with $c\geq 0$;
\item the Hessian of $\log \rho$ is uniformly bounded from above on $\R^d$;
\item $\sup_{|x|=r} \rho(x)^{1/2} r^{d-1}\to 0$ and $\sup_{|x|=r} |\nabla\rho(x)| \rho(x)^{-1/2} r^{d-1}\to 0$ as $r\to\infty$.
\end{enumerate}
Then there exists a constant $C>0$ such that the semigroup $T_\alpha(\cdot)$ generated by $-A^\alpha$ has a kernel $p_\alpha(t,x,y)$ which satisfies
\begin{equation}\label{eq: Cor: fractional kernel estimate example 1}
p_\alpha(t,x,y) \leq C t^{-\frac{d}{2\alpha}} e^{c^\alpha t} V(x) V(y)
\end{equation}
for all $t>0$ and $x,y\in\R^d$.
\end{corollary}

\begin{proof}
The idea is to apply Theorem \ref{Thm: fractional kernel estimate for 0<alpha<1} with $B(x)=x^\frac{2}{d}$.
We observe that $V=\rho^{-1/2}\notin L^2(\mu)$, so $V\notin D(A)$ and it is not possible to directly apply Theorem \ref{Thm: fractional kernel estimate for 0<alpha<1}. 
However, as pointed out in Remark \ref{rem: assumption V in the domain of A}, we only need to justify inequality \eqref{eq1: case alpha=1/2}. 

Let $t>0$ and $f$ be a smooth function with compact support.
As in the proof of \cite[Corollary 4.2]{Bakry}, we first prove that the following integration by parts holds for the operator $A$
\begin{equation}\label{eq3: Cor: fractional kernel estimate example 1}
\int_{\R^d} V(-AT(t)f)\, d\mu=\int_{\R^d} (-AV)T(t)f\, d\mu
\end{equation}
for any $t>0$ and any smooth function $f$ with compact support.
Let $r>0$, $B_r$ be the open ball of $\R^d$ of radius $r$ and center $0$ and $\vec{\eta}$ be its outward unit normal vector.
Integrating by parts on $B_r$ we obtain
\begin{align*}
\int_{B_r} V(-AT(t)f)\, d\mu
= &\int_{B_r} V\rho \Delta T(t)f \, dx + \int_{B_r} V \nabla\rho \cdot \nabla T(t)f\, dx\\
= &-\int_{B_r} \nabla(V\rho)\cdot \nabla T(t)f \, dx + \int_{\partial B_r} V \nabla T(t)f \cdot \vec{\eta}\rho\, d\sigma \\
&- \int_{B_r} \mathrm{div}(V \nabla\rho) T(t)f\, dx + \int_{\partial B_r} V T(t)f \nabla \rho\cdot\vec{\eta}\, d\sigma.
\end{align*}
Integrating by parts again the first integral in the right hand side of the previous expression we get
\begin{align*}
\int_{B_r} V(-AT(t)f)\, d\mu 
= &\int_{B_r} [\Delta (V\rho)-\mathrm{div}(V \nabla\rho)] T(t)f\, dx +\int_{\partial B_r} T(t)f[V\nabla\rho - \nabla (V\rho)]\cdot \vec{\eta}\, d\sigma\\
&+ \int_{\partial B_r} V \nabla T(t)f \cdot \vec{\eta}\rho\, d\sigma\\
=& \int_{B_r} (-AV)T(t)f\, d\mu -\int_{\partial B_r} T(t)f \nabla V\cdot \vec{\eta}\rho\, d\sigma+ \int_{\partial B_r} V \nabla T(t)f \cdot \vec{\eta}\rho\, d\sigma.
\end{align*}
We denote by $S^{d-1}$ the unit ball in $\R^{d-1}$ and performing a change of variables we find that
\begin{align}\label{eq2: Cor: fractional kernel estimate example 1}
\int_{B_r} V(-AT(t)f)\, d\mu 
= &\int_{B_r} (-AV)T(t)f\, d\mu -\int_{S^{d-1}} T(t)f(r\omega)\nabla V(r\omega)\cdot \vec{\eta}\rho(r\omega)r^{d-1}\, d\omega\notag\\
&+ \int_{S^{d-1}} V(r\omega) \nabla T(t)f(r\omega) \cdot \vec{\eta}\rho(r\omega)r^{d-1}\, d\omega\notag \\
=& \int_{B_r} (-AV)T(t)f\, d\mu -I_r+J_r,
\end{align}
where 
\begin{align}
I_r&=\int_{S^{d-1}} T(t)f(r\omega)\nabla V(r\omega)\cdot \vec{\eta}\rho(r\omega)r^{d-1}\, d\omega,\label{eq5: Cor: fractional kernel estimate example 1}\\
J_r&=\int_{S^{d-1}} V(r\omega) \nabla T(t)f(r\omega) \cdot \vec{\eta}\rho(r\omega)r^{d-1}\, d\omega.\label{eq6: Cor: fractional kernel estimate example 1}
\end{align}
We have
\[
|I_r|\leq \int_{S^{d-1}}|T(t)f(r\omega)| |\nabla V(r\omega)|\rho(r\omega)r^{d-1}\, d\omega
\leq \frac{1}{2} \norm{f}_\infty \int_{S^{d-1}} |\nabla \rho(r\omega)| \rho(r\omega)^{-1/2} r^{d-1}\, d\omega.
\]
Then by assumption (c) we derive that $I_r\to 0$ as $r\to\infty$. 
Moreover, since the Hessian of $\log\rho$ is uniformly bounded from above on $\R^d$, let say by a constant $\nu\in \R$, it follows that $-A$ satisfies the Bakry-Emery Curvature-Dimension condition $CD(-\nu ,\infty)$ and hence, cf. \cite[Proposition 5.4.1]{ANE}, one obtains
\[
|\nabla T(t)f| \leq e^{\nu t} T(t)|\nabla f|\leq e^{\nu t} \norm{\nabla f}_\infty. 
\]
Hence, arguing similarly, we deduce that $J_r\to 0$ as $r\to\infty$.
Consequently, the last two terms in \eqref{eq2: Cor: fractional kernel estimate example 1} tend to $0$ as $r$ tends to infinity. Thus, the integration by parts formula \eqref{eq3: Cor: fractional kernel estimate example 1} is valid.

We now prove that
\begin{equation}\label{eq7: Cor: fractional kernel estimate example 1}
\int_{\R^d} V(-A^\beta f)\, d\mu=\int_{\R^d} (-A^\beta V)f\, d\mu,
\end{equation}
for any $\beta\in (0,1)$. Let $\varphi$ be a cut-off function such that the support of $f$ is contained in the support of $\varphi$ and $\varphi =1$ on the support of $f$.
Since $C_c^\infty(\R^d)$ is a core for $-A$, we can apply the Balakrishnan representation formula \eqref{Balakrishnan formula} obtaining that
\begin{align*}
\int_{\R^d} V(-A^\beta f)\, d\mu =& \int_{\R^d} \varphi V(-A^\beta f)\, d\mu
=\frac{\sin \beta \pi}{\pi} \int_{\R^d} d\mu \int_0^{+\infty} s^{\beta-1}\varphi V (s+A)^{-1} (-Af)\, ds\\
=& \frac{\sin \beta \pi}{\pi} \int_{\R^d} d\mu  \int_0^{+\infty} s^{\beta-1} ds \int_0^{+\infty} e^{-st} \varphi VT(t)(-Af)dt\\
=& \frac{\sin \beta \pi}{\pi} \int_{\R^d} d\mu  \int_0^{+\infty} s^{\beta-1} ds \int_0^{+\infty} e^{-st} \varphi V(-AT(t)f)dt\\
=& \frac{\sin \beta \pi}{\pi} \int_0^{+\infty} s^{\beta-1} ds \int_0^{+\infty} e^{-st}dt\int_{\R^d} \varphi V(-AT(t)f)d\mu.
\end{align*}
Using the integration by parts formula \eqref{eq3: Cor: fractional kernel estimate example 1}, which remains true if we replace $V$ by $\varphi V$, and since the semigroup $T(\cdot)$ is symmetric, we derive that
\begin{align}\label{eq8: Cor: fractional kernel estimate example 1}
\int_{\R^d} V(-A^\beta f)\, d\mu
=& \frac{\sin \beta \pi}{\pi} \int_0^{+\infty} s^{\beta-1} ds \int_0^{+\infty} e^{-st}dt\int_{\R^d} (-A(\varphi V))T(t)f d\mu\notag\\
=& \frac{\sin \beta \pi}{\pi} \int_0^{+\infty} s^{\beta-1} ds \int_0^{+\infty} e^{-st}dt\int_{\R^d} f T(t) (-A(\varphi V))d\mu\notag\\
=& \frac{\sin \beta \pi}{\pi} \int_{\R^d}f d\mu  \int_0^{+\infty} s^{\beta-1} ds \int_0^{+\infty} e^{-st}T(t) (-A(\varphi V)) dt.
\end{align}
Here we observe that $\varphi V\in D(A)$, since $\varphi V\in C_c^2(\R^d)$ and one can see, from the proof of \cite[Corollary 3.7]{Albanese-Lorenzi-Mangino}, that $C_c^2(\R^d)$ is a core for $-A$.
Consequently, we can apply the Balakrishnan representation formula \eqref{Balakrishnan formula} and we deduce that
\begin{align*}
& \frac{\sin \beta \pi}{\pi} \int_{\R^d}f d\mu  \int_0^{+\infty} s^{\beta-1} ds \int_0^{+\infty} e^{-st}T(t) (-A(\varphi V)) dt\\
=& \frac{\sin \beta \pi}{\pi}\int_{\R^d} fd\mu \int_0^{+\infty} s^{\beta-1} (s+A)^{-1} (-A(\varphi V))\, ds\\
=&\int_{\R^d} (-A^\beta (\varphi V))f\, d\mu
= \int_{\R^d} (-A^\beta V)f\, d\mu.
\end{align*}
Combining this with \eqref{eq8: Cor: fractional kernel estimate example 1} leads to \eqref{eq7: Cor: fractional kernel estimate example 1}.

Moreover, since $C_c^\infty(\R^d)$ is a core for $A$ and $D(A)$ is a core for $A^\beta$ by \cite[Proposition 3.1.1]{Haase}, then \eqref{eq7: Cor: fractional kernel estimate example 1} holds for any $f\in D(A^\beta)$. 
Finally, applying \eqref{eq7: Cor: fractional kernel estimate example 1} with $T_\beta(t)f$ instead of $f$ we find that
\begin{equation}\label{eq4: Cor: fractional kernel estimate example 1}
\int_{\R^d} V(-A^\beta T_\beta (t)f)\, d\mu=\int_{\R^d} (-A^\beta V)T_\beta (t)f\, d\mu,
\end{equation}
for any $\beta \in (0,1)$, $t>0$ and any smooth function $f$ with compact support.

It is now possible to apply Theorem \ref{Thm: fractional kernel estimate for 0<alpha<1} to infer that
\begin{equation}\label{eq1: Cor: fractional kernel estimate example 1}
p_\alpha (t,x,y)\leq K(t)^2 e^{c^\alpha t} V(x)V(y),
\end{equation}
for all $t>0$ and $x,y\in\R^d$, where $K$ is defined as in \eqref{def: function K}.
Computing the function $U(x)$ defined by \eqref{eq1: norm semigroup for 0<alpha<1} we have
\[
U(x)=\int_x^\infty \frac{1}{\gamma u [B(\gamma u)]^\alpha}\, du
= \int_x^\infty (\gamma u)^{-1-\frac{2\alpha}{d}}\, du=\frac{d}{2\alpha \gamma^{1+\frac{2\alpha}{d}}} x^{-\frac{2\alpha}{d}},
\]
for all $x>0$. Moreover, since $U(0)=+\infty$, the function $K(x)$ in \eqref{def: function K} is given by
\[
K(x)=\sqrt{U^{-1}(x)} = C x^{-\frac{d}{4\alpha}},
\]
for some $C>0$.
One derives \eqref{eq: Cor: fractional kernel estimate example 1} from \eqref{eq1: Cor: fractional kernel estimate example 1} by plugging in the definition of $K$.
\end{proof}

We now illustrate the previous result on the examples of measures of Cauchy and exponential types.

\begin{corollary}\label{Cor: example polynomial case}
Let $0<\alpha<1$ and $p_\alpha$ be the integral kernel associated with $A^\alpha$ with $\rho(x)=(1+|x|^2)^{-\beta}$ and $\beta>d$. Then 
\[
p_\alpha(t,x,y) \leq C t^{-\frac{d}{2\alpha}} e^{(\beta d)^\alpha t} (1+|x|^2)^\frac{\beta}{2} (1+|y|^2)^\frac{\beta}{2},
\]
for any $t>0$ and $x,y\in\R^d$, where $C$ is a positive constant.
\end{corollary}

\begin{proof}
We check the assumptions of Corollary \ref{Cor: fractional kernel estimate example 1}. Since $\beta>d$, then $d\mu(x)=\rho(x)dx$ is a probability measure up to a positive constant. Moreover, since 
\[
\int_{\R^d} \rho(x)^{-1/2} d\mu(x)=\int_{\R^d} (1+|x|^2)^{-\frac{\beta}{2}}\, dx<\infty,
\]
then the function 
\[
V(x)=\rho(x)^{-1/2}=(1+|x|^2)^\frac{\beta}{2}
\]
belongs to $L^1(\mu)$. 
We now verify that $-AV\leq cV$ for some $c\geq 0$. Let $x\in \R^d$. By straightforward computations we have
\begin{align*}
\nabla V(x)&=\beta (1+|x|^2)^{-1} x V(x),\\
\Delta V(x)&= \beta(1+|x|^2)^{-2} [(d+\beta-2)|x|^2+d]V(x),\\
\nabla \log\rho(x)&=-2\beta (1+|x|^2)^{-1} x.
\end{align*}
Hence, we obtain
\begin{align*}
-AV(x)
&= \Delta V(x) + \nabla \log\rho(x)\cdot\nabla V(x)\\
&= \beta(1+|x|^2)^{-2} [(d+\beta-2)|x|^2+d]V(x)-2\beta^2 (1+|x|^2)^{-2} |x|^2 V(x)\\
&=\beta(1+|x|^2)^{-2} [(d-\beta-2)|x|^2+d]V(x)\\
&=\left[\beta d (1+|x|^2)^{-1} -\beta(\beta+2)(1+|x|^2)^{-2} |x|^2\right]V(x).
\end{align*}
Since $(1+|x|^2)^{-1}\leq 1$ and the second term in the square brackets in the right hand side is negative, it leads to
\[
-AV(x)\leq \beta d V(x).
\]
It implies that we can choose $c=\beta d\geq 0$ as Lyapunov constant. From the above computations it is also clear that $AV\in L^1(\mu)$. 
This proves assumption (a) in Corollary \ref{Cor: fractional kernel estimate example 1}.

The next step is to show (b). Let $x\in\R^d$. First we observe that, for $i,j=1,\dots,d$, we have
\[
D_{ij}\log\rho(x)=4\beta (1+|x|^2)^{-2} x_ix_j-2\beta (1+|x|^2)^{-1}\delta_{ij},
\]
where $\delta_{ij}$ is the Kronecker delta. Then, for $\xi\in\R^d$, we get
\begin{align*}
\sum_{i,j=1}^d D_{ij}\log\rho(x)\xi_i\xi_j
&= 4\beta (1+|x|^2)^{-2}\left(\sum_{i=1}^d  x_i\xi_i\right)^2-2\beta  (1+|x|^2)^{-1} \sum_{i=1}^d \xi_i^2\\
&\leq 4\beta (1+|x|^2)^{-2} |x|^2 |\xi|^2 -2\beta  (1+|x|^2)^{-1} |\xi|^2.
\end{align*}
Given that $(1+|x|^2)^{-2} |x|^2\leq 1$ and that the last term in the right hand side of the previous inequality is negative, we deduce that
\[
\sum_{i,j=1}^d D_{ij}\log\rho(x)\xi_i\xi_j \leq 4\beta |\xi|^2.
\]
Hence, the Hessian of $\log \rho$ is uniformly bounded from above on $\R^d$ by the constant $4\beta$.

Finally, we check (c). For $x\in\R^d$, we have
\begin{align*}
\rho(x)^{1/2} r^{d-1}&= (1+|x|^2)^{-\frac{\beta}{2}}r^{d-1},\\
|\nabla\rho(x)| \rho(x)^{-1/2} r^{d-1}&= 2\beta (1+|x|^2)^{-\frac{\beta+2}{2}} |x| r^{d-1}.
\end{align*}
Therefore, since $\beta>d$, condition (c) is satisfied.

Then the statement follows by inequality \eqref{eq: Cor: fractional kernel estimate example 1} considering that $c=\beta d$ and $V(x)=(1+|x|^2)^\frac{\beta}{2}$.
\end{proof}

\begin{corollary}\label{Cor: example exponential case}
Let $0<\alpha<1$ and $p_\alpha$ be the integral kernel associated with $A^\alpha$.
If $\rho(x)=e^{-(1+|x|^2)^{a/2}}$ and $0<a<2$, then
\begin{equation}\label{eq3: Cor: kernel estimate exponential measure}
p_\alpha(t,x,y) \leq C t^{-\frac{d}{2\alpha}} e^{c^\alpha t} e^\frac{(1+|x|^2)^{a/2}}{2} e^\frac{(1+|y|^2)^{a/2}}{2},
\end{equation}
for any $t>0$ and $x,y\in\R^d$.
If $\rho(x)=e^{-|x|^a}$ and $a\geq 2$, then 
\begin{equation}\label{eq4: Cor: kernel estimate exponential measure}
p_\alpha(t,x,y) \leq C t^{-\frac{d}{2\alpha}} e^{c^\alpha t} e^\frac{|x|^a}{2} e^\frac{|y|^a}{2},
\end{equation}
for any $t>0$ and $x,y\in\R^d$.
Here $C$ is a positive constant. Moreover, $c=\frac{1}{2} ad$ if $a\leq 2$ and $c=\frac{1}{2}a(a+d-2)K^{a-2}$ with $K\geq  \left( \frac{2(a+d-2)}{a}\right)^{1/a}$ if $a\geq 2$.
\end{corollary}

\begin{proof}
In view of applying Corollary \ref{Cor: fractional kernel estimate example 1}, we verify conditions (a), (b) and (c).
In both cases, up to a positive constant, $d\mu(x)=\rho(x)dx$ is a probability measure and $V(x)=\rho(x)^{-1/2}$ belongs to $L^1(\mu)$.
We now distinguish between the two cases.

\smallskip

\textit{Case 1}: $\rho(x)=e^{-(1+|x|^2)^{a/2}}$ and $0<a<2$.

We check that $-AV\leq cV$ for some $c\geq 0$ to show that (a) holds. Let $x\in \R^d$. Easy computations lead to
\begin{align*}
\nabla V(x)&=\frac{1}{2} a (1+|x|^2)^\frac{a-2}{2} x V(x),\\
\Delta V(x)&= \left[\frac{1}{2} a(a-2)(1+|x|^2)^\frac{a-4}{2}|x|^2+\frac{1}{2} ad(1+|x|^2)^\frac{a-2}{2}+\frac{1}{4} a^2(1+|x|^2)^{a-2}|x|^2\right]V(x),\\
\nabla \log\rho(x)&=-a (1+|x|^2)^\frac{a-2}{2} x.
\end{align*}
So we get
\begin{align}\label{eq1: Cor: kernel estimate exponential measure}
-AV(x)
&= \Delta V(x) + \nabla \log\rho(x)\cdot\nabla V(x)\notag\\
&= \left[\frac{1}{2} a(a-2)(1+|x|^2)^\frac{a-4}{2}|x|^2+\frac{1}{2} ad(1+|x|^2)^\frac{a-2}{2}-\frac{1}{4} a^2(1+|x|^2)^{a-2}|x|^2\right]V(x).
\end{align}
Since $0<a<2$, the first term in the right hand side of \eqref{eq1: Cor: kernel estimate exponential measure} is negative together with the third one and $(1+|x|^2)^\frac{a-2}{2}\leq 1$. Hence, we deduce that
\[
-AV(x) \leq \frac{1}{2} ad V(x).
\]
We then choose $c=\frac{1}{2} ad$ as Lyapunov constant.

We now prove (b), namely that the Hessian of $\log \rho$ is uniformly bounded from above on $\R^d$. Let $x\in\R^d$. For $i,j=1,\dots,d$ we have
\[
D_{ij}\log\rho(x)=a(2-a)(1+|x|^2)^\frac{a-4}{2} x_ix_j-a(1+|x|^2)^\frac{a-2}{2}\delta_{ij}.
\]
Thus, we deduce 
\begin{align*}
\sum_{i,j=1}^d D_{ij}\log\rho(x)\xi_i\xi_j
&= a(2-a)(1+|x|^2)^\frac{a-4}{2}\left(\sum_{i=1}^d  x_i\xi_i\right)^2-a(1+|x|^2)^\frac{a-2}{2} |\xi|^2\\
&\leq a(2-a)(1+|x|^2)^\frac{a-4}{2}|x|^2|\xi|^2\leq a(2-a)|\xi|^2,
\end{align*}
for any $\xi\in\R^d$. It leads to condition (b).

Moreover, for $x\in\R^d$, we find that
\begin{align*}
\rho(x)^{1/2} r^{d-1}&= e^{-\frac{(1+|x|^2)^{a/2}}{2}}r^{d-1},\\
|\nabla\rho(x)| \rho(x)^{-1/2} r^{d-1}&= a (1+|x|^2)^{\frac{a-2}{2}} |x| e^{-\frac{(1+|x|^2)^{a/2}}{2}} r^{d-1}.
\end{align*}
From that it is easy to see that condition (c) is also verified.
Finally, applying Corollary \ref{Cor: fractional kernel estimate example 1}, inequality \eqref{eq3: Cor: kernel estimate exponential measure} follows.

\smallskip

\textit{Case 2}: $\rho(x)=e^{-|x|^a}$ and $a\geq 2$.

For all $x\in\R^d$ we have that
\begin{equation}\label{eq2: Cor: kernel estimate exponential measure}
-AV(x)
= \Delta V(x) + \nabla \log\rho(x)\cdot\nabla V(x)
= \frac{1}{2} a |x|^{2a-2} V(x)\left[(a+d-2) |x|^{-a}-\frac{1}{2} a \right].
\end{equation}
Assume that $|x|\geq K$ for some $K>0$. Since $a+d-2>0$, we get
\[
(a+d-2)|x|^{-a}-\frac{1}{2} a\leq (a+d-2)K^{-a}-\frac{1}{2} a.
\]
Hence, the quantity in the square brackets in the right hand side of \eqref{eq2: Cor: kernel estimate exponential measure} is negative if we take 
\[
K\geq  \left( \frac{2(a+d-2)}{a}\right)^{1/a}.
\]
We deduce that $-AV(x)\leq 0$ for any $x\in\R^d$ such that $|x|\geq K$.
For the remaining values of $x$, $|x|<K$, we drop the second term in the square brackets in the right hand side of \eqref{eq2: Cor: kernel estimate exponential measure} because it's negative, obtaining that
\[
-AV(x)\leq \frac{1}{2}a(a+d-2)K^{a-2} V(x).
\]
We conclude that $-AV(x)\leq cV(x)$ for all $x\in\R^d$,  where $c=\frac{1}{2}a(a+d-2)K^{a-2}$.
This proves condition (a) in Corollary \ref{Cor: fractional kernel estimate example 1}.

We now compute the second partial derivatives of $\log\rho$ for any $i,j=1,\dots,d$ and $x\in\R^d$
\[
D_{ij}\log\rho(x)=a(2-a)|x|^{a-4} x_ix_j-a|x|^{a-2}\delta_{ij}.
\]
Then, since $a\geq 2$, we have that
\[
\sum_{i,j=1}^d D_{ij}\log\rho(x)\xi_i\xi_j=a(2-a)|x|^{a-4} \left(\sum_{i=1}^d  x_i\xi_i\right)^2-a|x|^{a-2} |\xi|^2\leq 0,
\]
for all $\xi\in\R^d$. This yields (b).

Lastly, for $x\in\R^d$ we have that
\begin{align*}
\rho(x)^{1/2} r^{d-1}&= e^{-\frac{|x|^a}{2}}r^{d-1},\\
|\nabla\rho(x)| \rho(x)^{-1/2} r^{d-1}&= a |x|^{a-1} e^{-\frac{|x|^a}{2}} r^{d-1}.
\end{align*}
From that we deduce that (c) holds.
Then inequality \eqref{eq4: Cor: kernel estimate exponential measure} follows by Corollary \ref{Cor: fractional kernel estimate example 1}.
\end{proof}

\begin{remark}
The operator $A$ in $L^2(\mu)$ is equivalent to the Schr\"odinger operator $B=-\Delta -U$ in $L^2(\R^d)$, where
\[
U:=\frac{1}{4} \left| \frac{\nabla \rho}{\rho}\right|^2 -\frac{1}{2} \frac{\Delta \rho}{\rho}.
\]
Indeed, if we consider the transformation $T:L^2(\R^d)\to L^2(\mu)$ defined by $Tf=\frac{1}{\sqrt{\rho}}f$ for any $f\in  L^2(\R^d)$, we have that $B=T^{-1} A T$.
Moreover, computing $B^\alpha$ for $0<\alpha<1$ via the Balakrishnan representation formula \eqref{Balakrishnan formula}, it turns out that
\[
B^\alpha=T^{-1} A^\alpha T.
\]
So, the fractional operator $A^\alpha$ in $L^2(\mu)$ is equivalent to $(-\Delta-U)^\alpha$ in $L^2(\R^d)$.
We also observe that the integral kernel $k_\alpha(t,x,y)$ associated with the operator $-B^\alpha$, is given by
\begin{equation}\label{eq: rem: kernel Schrodinger operator}
k_\alpha(t,x,y)= \sqrt{\rho(x)\rho(y)}\,p_\alpha(t,x,y)
\end{equation}
for any $t>0$ and $x,y\in\R^d$.
Consequently, applying the results above, we find estimates for the kernel associated with the fractional power $B^\alpha$ of the Schr\"odinger operator $B$ in $L^2(\R^d)$ for any $0<\alpha<1$.
For instance, if $\rho(x)=(1+|x|^2)^{-\beta}$ with $\beta>d$, then $A$ is equivalent to the Schr\"odinger operator
\[
B=-\Delta +\beta(\beta+2) \frac{|x|^2}{(1+|x|^2)^{2}} -\frac{\beta d}{1+|x|^2}.
\]
Then, by \eqref{eq: rem: kernel Schrodinger operator} and Corollary \ref{Cor: example polynomial case}, the integral kernel $k_\alpha$ associated with $B^\alpha$ satisfies
\[
k_\alpha(t,x,y) \leq C t^{-\frac{d}{2\alpha}} e^{(\beta d)^\alpha t},
\]
for any $t>0$ and $x,y\in\R^d$, where $C$ is a positive constant.
Similarly, in the case $\rho(x)=e^{-|x|^a}$ with $a\geq 2$,  $A$ is equivalent to the Schr\"odinger operator
\[
B=-\Delta - \frac{1}{2} a(a+d-2) |x|^{a-2}+\frac{1}{4} a^2 |x|^{2a-2}.
\]
If we now apply Corollary \ref{Cor: example exponential case}, then we find the following estimate for the integral kernel $k_\alpha$ associated with $B^\alpha$
\[
k_\alpha(t,x,y) \leq C t^{-\frac{d}{2\alpha}} e^{c^\alpha t}
\]
for any $t>0$ and $x,y\in\R^d$, where $C$ is a positive constant and $c$ is defined as in the statement of Corollary \ref{Cor: example exponential case}.
\end{remark}

\end{document}